\documentclass[a4paper,10pt]{amsart}

\usepackage{amsthm,amsmath,amssymb}
\usepackage[T1]{fontenc}
\usepackage[english]{babel}
\usepackage{mathscinet}
\usepackage{graphics}
\usepackage{graphicx}
\usepackage[all,cmtip]{xy}
\usepackage{multicol}
\usepackage{multirow}
\usepackage{algorithm}
\usepackage{algorithmic}
\usepackage{type1ec}
\usepackage{type1cm}
\usepackage{fix-cm}
\usepackage[lowtilde]{url}
\usepackage{hyperref}
\usepackage{enumerate}

\theoremstyle{plain} 
\newtheorem{theo}{Theorem}[section]  
\newtheorem{prop}[theo]{Proposition}

\newtheorem{lemma}[theo]{Lemma}

\theoremstyle{definition}
\newtheorem{defi}[theo]{Definition}
\newtheorem{ex}[theo]{Example}

\theoremstyle{remark}
\newtheorem{remark}[theo]{Remark}
\newtheorem{nothing}[theo]{\noindent\!\!\bf}

\DeclareMathOperator{\lcm}{lcm}
\DeclareMathOperator{\Sing}{Sing}

\DeclareMathOperator{\ord}{ord}

\DeclareMathOperator{\Mat}{Mat}

\DeclareMathOperator{\WeDiv}{WeDiv}
\DeclareMathOperator{\CaDiv}{CaDiv}
\DeclareMathOperator{\Div}{Div}
\DeclareMathOperator{\Pic}{Pic}
\DeclareMathOperator{\Cl}{Cl}
\DeclareMathOperator{\codim}{codim}
\DeclareMathOperator{\Supp}{Supp}
\DeclareMathOperator{\length}{length}

\DeclareMathOperator{\diag}{Diag}

\newcommand\Z{\mathbb{Z}}
\newcommand\N{\mathbb{N}}
\newcommand\C{\mathbb{C}}
\newcommand\Q{\textbf{Q}}

\newcommand\bP{\mathbb{P}}
\newcommand\w{\omega}

\newcommand\cO{\mathcal{O}}
\newcommand\bd{\mathbf{d}}
\newcommand\be{\mathbf{e}}
\newcommand\ba{\mathbf{a}}
\newcommand\bxi{\boldsymbol{\xi}}
\title{Cartier and Weil Divisors on Varieties with Quotient Singularities}

\author[E.~Artal]{Enrique Artal Bartolo}
\address[E.~Artal]{Departamento de Matem\'{a}ticas-IUMA  \\
Universidad de Zaragoza \\
C/~Pedro Cerbuna 12, 50009, Zaragoza, Spain}
\email{artal@unizar.es}

\author[J.~Mart\'{i}n-Morales]{Jorge Mart\'{i}n-Morales}
\address[J.~Mart\'{i}n-Morales]{Centro Universitario de la Defensa-IUMA \\
Academia General Militar \\
Ctra.~de Huesca s/n. 50090, Zaragoza, Spain}
\email{jorge@unizar.es}
\urladdr{http://cud.unizar.es/martin}

\author[J.~Ortigas-Galindo]{Jorge Ortigas-Galindo}
\address[J.~Ortigas-Galindo]{Departamento de Matem\'{a}ticas-IUMA  \\
Universidad de Zaragoza \\
C/~Pedro Cerbuna 12, 50009, Zaragoza, Spain}
\email{jortigas@unizar.es}
\urladdr{\url{http://riemann.unizar.es/\~jortigas/}}

\date{\today}

\begin{document}

\begin{abstract}
The main goal of this paper is to show
that the notions of Weil and Cartier $\mathbb{Q}$-divisors coincide for
$V$-varieties and give a procedure to express a rational Weil Divisor as a
rational Cartier divisor.
The theory is illustrated with weighted projective spaces
and weighted blow-ups.
\end{abstract}


\maketitle
\setcounter{tocdepth}{1}
\tableofcontents

\section*{Introduction}

Singularity Theory deals with the study of complex spaces with singularities.
The most powerful tool to study singularities is resolution,
i.e., a proper analytic morphism from a smooth variety which is
an isomorphism outside the singular locus. The existence of such
resolutions are guaranteed by the work of Hironaka. Usually
the combinatorics of the exceptional locus of the resolution is quite complicated.

There is a class of singularities which is well-understood, namely normal singularities
which are obtained as the quotient of a ball in $\C^n$ by the linear action of a finite
group. Spaces admitting only such singularities are called $V$-varieties and were introduced
by Bailly. These varieties share a lot of properties with smooth ones, e.g., 
projective $V$-varieties also carry a Hodge structure and a natural notion of
normal crossing divisors can be defined (called $\mathbb{Q}$-normal crossing divisors).
As it will be developed in the Ph.D. thesis of the second author, the study
of so-called $\Q$-resolutions (allowing quotient singularities, specially Abelian)
provides a better comprehension of some families of singularities.

Intersection theory is needed for this study and we need to deal with divisors
in $V$-varieties. Two kind of divisors appear in the litterature: Weil and Cartier
divisors. Weil divisors are formal sums of hypersurfaces and Cartier divisors
are sections of the quotient sheaf of meromorphic functions modulo non-vanishing
holomorphic functions. The relationship between Cartier divisors and line bundles
provide a nice way to define the intersection of two divisors. In the smooth category,
the two notions coincide but it is not the case for singular ones (even for normal varities).
One can also consider Weil and Cartier $\mathbb{Q}$-divisors (tensorizing the corresponding
groups by  $\mathbb{Q}$). The main result of this paper is that these two notions coincide
for $V$-varieties and hence an intersection theory can be developed for Weil $\mathbb{Q}$-divisors
using the theory of line bundles.

This result is probably known for specialists but we have not found a proof in the litterature.
There are some partial results for toric varities (defined with simplicial cones). Moreover,
in this work we give an algorithm to represent a Weil $\mathbb{Q}$-divisor as a Cartier
$\mathbb{Q}$-divisor. This example will be illustrated in some spaces obtained
by weighted blow-ups. These blow-ups can be understood from toric geometry but in this
work we present in a more geometric way, generalizing the presentation of standard blow-ups.
For this definition, we need to give a presentation of abelian quotient singularities
In a forthcoming work, we will develop an intersection theory over surfaces with
abelian quotient singularities.

The paper is organized as follows. In \S\ref{sec-vq} we give a general presentation of varieties 
with quotient singularities and list their basic properties. In \S\ref{weighted_projective_space} we deal with
weighted projective spaces and in \S\ref{sec-wbr} we introduced weighted-blow-ups.
The main results about rational Weil and Cartier divisors on $V$-varieties  are stated and proved 
in \S\ref{sec-cw}. We finish the work with some conlusions and the expected future work.

{\bf Acknowledgements.}
We wish to express our gratitude to J.I. Cogolludo for his support. All authors are partially supported by
the Spanish projects MTM2010-2010-21740-C02-02 and ``E15 Grupo Consolidado
Geometr{\'i}a'' from the government of Arag{\'o}n. Also, the second author,
J.~Mart{\'i}n-Morales, is partially supported by FQM-333, Junta de Andaluc{\'i}a.

\section{$V$-Manifolds and Quotient Singularities}\label{sec-vq}

\begin{defi}
A $V$-manifold of dimension $n$ is a complex analytic space which admits an open
covering $\{U_i\}$ such that $U_i$ is analytically isomorphic to $B_i/G_i$ where
$B_i \subset \C^n$ is an open ball and $G_i$ is a finite subgroup of $GL(n,\C)$.
\end{defi}
 
The concept of $V$-manifolds was introduced in~\cite{Satake56} and they have the same homological
properties over $\mathbb{Q}$ as manifolds. For instance, they admit a Poincar{\'e}
duality if they are compact and carry a pure Hodge structure if they are compact
and Kähler, see~\cite{Baily56}. They have been locally classified by
Prill~\cite{Prill67}. To state this local result we need the following.

\begin{defi}
A finite subgroup $G$ of $GL(n,\C)$ is called {\em small} if no element of $G$
has $1$ as an eigenvalue of multiplicity precisely $n-1$, that is, $G$ does not
contain rotations around hyperplanes other than the identity.
\end{defi}

For every finite subgroup $G$ of $GL(n,\C)$ denote by $G_{\text{big}}$ the
normal subgroup of $G$ generated by all rotations around hyperplanes. Then the
$G_{\text{big}}$-invariant polynomials form a polynomial algebra and hence
$\C^n/G_{\text{big}}$ is isomorphic to $\C^n$.

The group $G/G_{\text{big}}$ maps isomorphically to a small subgroup of
$GL(n,\C)$, once a basis of invariant polynomials has been chosen. Hence the
local classification of $V$-manifolds reduces to the classification of actions
of small subgroups of $GL(n,\C)$.

\begin{theo}\label{th_Prill}{\rm (\cite{Prill67}).}~Let $G_1$, $G_2$ be small
subgroups of $GL(n,\C)$. Then $\C^n/G_1$ is isomorphic to $\C^n/G_2$ if and only
if $G_1$ and $G_2$ are conjugate subgroups. $\hfill \Box$
\end{theo}

We are interested in $V$-manifolds where the quotient spaces $B_i / G_i$ are
given by (finite) abelian groups. In this case the following notation is used.

\begin{nothing}\label{notation_action}
For $\bd: = {}^t(d_1 \ldots d_r)$ we denote 
$\mu_{\bd} := \mu_{d_1} \times \cdots \times \mu_{d_r}$ a finite
abelian group written as a product of finite cyclic groups, that is, $\mu_{d_i}$
is the cyclic group of $d_i$-th roots of unity in $\C$. Consider a matrix of weight
vectors 
$$
A := (a_{ij})_{i,j} = [\ba_1 \, | \, \cdots \, | \, \ba_n ] \in\Mat (r
\times n, \Z),\quad
\ba_i:={}^t (a_{1 i}\dots a_{r i})\in\Mat(r\times 1,\Z),
$$ 
and the action
\begin{equation}\label{action_XdA}
\begin{array}{cr}
( \mu_{d_1} \times \cdots \times \mu_{d_r} ) \times \C^n  \longrightarrow  \C^n,&\bxi_\bd := (\xi_{d_1}, \ldots, \xi_{d_r}),
\\[0.15cm]
\big( \bxi_{\bd} , \mathbf{x} \big)  \mapsto  (\xi_{d_1}^{a_{11}} \cdot\ldots\cdot
\xi_{d_r}^{a_{r1}}\, x_1,\, \ldots\, , \xi_{d_1}^{a_{1n}}\cdot\ldots\cdot
\xi_{d_r}^{a_{rn}}\, x_n ),&
\mathbf{x} := (x_1,\ldots,x_n).
\end{array}
\end{equation}
Note that the $i$-th row of the matrix $A$ can be considered modulo $d_i$. The
set of all orbits $\C^n / G$ is called ({\em cyclic}) {\em quotient space of
type $(\bd;A)$} and it is denoted by
$$
  X(\bd; A) := X \left( \begin{array}{c|ccc} d_1 & a_{11} & \cdots & a_{1n}\\
\vdots & \vdots & \ddots & \vdots \\ d_r & a_{r1} & \cdots & a_{rn} \end{array}
\right).
$$
The orbit of an element $\mathbf{x}\in \C^n$ under this action is denoted
by $[\mathbf{x}]_{(\bd; A)}$ and the subindex is omitted if no ambiguity
seems likely to arise. With multi-index notation
the action takes the simple form
\begin{eqnarray*}
\mu_\bd \times \C^n & \longrightarrow & \C^n, \\
(\bxi_\bd, \mathbf{x}) & \mapsto & \bxi_\bd\cdot\mathbf{x}:=(\bxi_\bd^{\ba_1}\, x_1, \ldots, \bxi_{\bd}^{\ba_n}\, x_n).
\end{eqnarray*}
\end{nothing}

The following classic result shows that the family of varieties which can locally be
written like~$X(\bd; A)$ is exactly the same as the family of $V$-manifolds with
abelian quotient singularities.

\begin{lemma}\cite{Prill67}
Let $G$ be a finite abelian subgroup of $GL(n,\C)$. Then $\C^n / G$ is
isomorphic to some quotient space of type $(\bd;A)$.
\end{lemma}

\begin{proof}
The group $G$ can be seen as a direct product
of cyclic groups generated by matrices $M_1, \ldots, M_r$ of order~$d_1,\dots,d_r$
such that 
$$
G = \{ M_1^{i_1} \cdots M_r^{i_r} \mid i_k = 0,\ldots, d_k-1\}.
$$
Each of these matrices $M_i$, $i = 1, \ldots, r$, is conjugated to 
$\diag(\zeta_{d_i}^{a_{i1}},\dots,\zeta_{d_i}^{a_{in}})$
where $\zeta_{d_i}$ is a $d_i$-th primitive root of unity. Moreover, they are
simultaneously diagonalizable because they commute. This proves $\C^n / G \simeq
X(\bd;(a_{ij})_{i,j})$.
\end{proof}

Different types $(\bd;A)$ can give rise to isomorphic quotient spaces, see
Remark~\ref{operation_XdA}. We shall prove that they can always be represented
by an upper triangular matrix of dimension $(n-1) \times n$, see
Lemmas~\ref{movs-dA} and~\ref{upper_triangular}. Finding a simpler type $(\bd;A)$ to represent a
quotient space will lead us to the notions of normalized type
and space, see~Definition~\ref{def_normalized_XdA}.

\begin{remark}\label{operation_XdA}
For $n=3$ the simple group automorphism on $\mu_d \times \mu_d$ given by
$(\xi,\eta) \mapsto (\xi \eta^{-1}, \eta)$ shows that the following two spaces
are isomorphic under the identity map:
$$
X \left(\begin{array}{c|ccc}
d & a_{11} & a_{12} & a_{13}\\
d & a_{21} & a_{22} & a_{23}
\end{array}\right) =
X \left(\begin{array}{c|ccc}
d & a_{11} & a_{12} & a_{13}\\
d & a_{21} - a_{11} & a_{22} - a_{12} & a_{23} - a_{13}
\end{array}\right).
$$
Note that the determinants of the minors of order $2$ are the same in both side
of the previous equation. There is an obvious generalization for 
higher dimensions, allowing row operations when the two rows correspond 
to the same cyclic groups. This is not a strong condition since we
can multiply also any row (of $\bd$ and $A$ simultaneously).
\end{remark}

\begin{ex}\label{Ex_quo_dim1}
When $n=1$ all spaces $X(\bd;A)$ are
isomorphic to~$\C$. 
Dividing each row by $\gcd(d_i, a_{i})$ do not change the quotient
and we can work with $X({}^t(d_1,\ldots,d_r); {}^t(a_{1},\ldots,a_{r}))$ 
where
$\gcd(d_i, a_{i})=1$.

The map $[x] \mapsto x^{d_1}$ gives an isomorphism between $X(d_1; a_{1})$ and
$\C$. For $r=2$ one has that (we write the symbol $=$ when the isomorphism is
induced by the identity map)
$$
\begin{array}{crclcrcl}
\displaystyle \frac{\C}{\mu_{d_1} \times \mu_{d_2}} = & \displaystyle \frac{\,
\C / \mu_{d_1} \,}{\mu_{d_2}} & \cong & \C / \mu_{d_2} & \stackrel{(*)}{=} &
X(d_2; a_{2} d_1) & \cong & \C,\\
& [x] & \mapsto & x^{d_1}, & & [x] & \mapsto & x^{\frac{d_2}{\gcd(d_1,d_2)}}.
\end{array}
$$

To see the equality $(*)$ observe that
$
\xi_{d_2} \cdot x^{d_1} \equiv \xi_{d_2} \cdot [x] = [\xi_{d_2}^{a_{2}} x]
\equiv \xi_{d_2}^{a_{2} d_1} x^{d_1}.
$
It follows that the corresponding quotient space is isomorphic to $\C$ under the
map $[x] \mapsto x^{\lcm (d_1,d_2 )}$.
\end{ex}

From Remark~\ref{operation_XdA} and Example~\ref{Ex_quo_dim1}
the following lemmas follow easily.

\begin{lemma}\label{movs-dA}
The following operations do not change the isomorphism type of $X(\bd;A)$
\begin{enumerate}[\rm(1)]
\item Permutation $\sigma$ of columns of $A$, $[\mathbf{x}]\mapsto[(x_{\sigma(1)},\dots,x_{\sigma(n)})]$.
\item Permutation of rows of $(\bd;A)$, $[\mathbf{x}]\mapsto[\mathbf{x}]$.
\item Multiplication of a row of $(\bd;A)$ by a positive integer, $[\mathbf{x}]\mapsto[\mathbf{x}]$.
\item Multiplication of a row of $A$ by an integer coprime with the
corresponding row in $\bd$, $[\mathbf{x}]\mapsto[\mathbf{x}]$.
\item Replace $a_{i j}$ by $a_{i j}+k d_j$, $[\mathbf{x}]\mapsto[\mathbf{x}]$.
\item If $e$ is coprime with $a_{i n}$ and divides $d_1$ and $a_{1,j}$, $1\leq j<n$, then replace,
$a_{i,n}\mapsto e a_{i,n}$, $[\mathbf{x}]\mapsto[(x_1,\dots,x_n^e)]$.
\item If $d_r=1$ then 
eliminate the last row, $[\mathbf{x}]\mapsto[\mathbf{x}]$.
\end{enumerate}
 
\end{lemma}

\begin{lemma}\label{upper_triangular}
The space $X(\bd;A) = \C^n / \mu_\bd$ can always be represented by an upper
triangular matrix of dimension $(n-1) \times n$. More precisely, there exist a
vector $\be := {}^t(e_1,\ldots, e_{n-1})$, a~matrix $B = (b_{i,j})_{i,j}$, and an
isomorphism $[(x_1,\ldots,x_n)] \mapsto [(x_1,\ldots, x_n^{k})]$ for some
$k\in\N$ such that 
$$
X(\bd; A) \cong \left(\begin{array}{c|cccc}
e_1 & b_{1,1} & \cdots & b_{1,n-1} & b_{1,n} \\
\vdots & \vdots & \ddots & \vdots & \vdots\\
e_{n-1} & 0 & \cdots & b_{n-1,n-1} & b_{n-1,n}
\end{array}\right) = X(\be;B). 
$$
\end{lemma}

\begin{nothing}\label{reduce_free}
The action showed in the equation~(\ref{action_XdA}) of \ref{notation_action} 
is free on $(\C^{*})^n$, 
i.e~$\big[ \mathbf{x} \in (\C^{*})^n, \ 
\bxi_\bd \cdot \mathbf{x} = \mathbf{x} \big] \Longrightarrow \,\bxi_\bd= 1$,
if and only if the group homomorphism $\mu_\bd \rightarrow GL(n,\C)$ given by
\begin{equation}\label{map_action}
\bxi_\bd = ( \xi_{d_1}, \ldots, \xi_{d_r} ) \longmapsto 
\diag(\bxi_d^{\ba_1}, \dots ,\bxi_d^{\ba_r})
\end{equation}
is injective. If this is not the case, let $H$ be the kernel of this group
homomorphism. Then $\C^n / H \equiv \C^n$ and the group $\mu_d / H$ acts freely
on $(\C^{*})^n$ under the previous identification. Thus one can always assume
that the free and the small conditions are satisfied. This motivates the
following definition.
\end{nothing}

\begin{defi}\label{def_normalized_XdA}
The type $(\bd;A)$ is said to be {\em normalized} if the following two conditions
hold.
\begin{enumerate}
\item The action is free on $(\C^{*})^n$.
\item The group~$\mu_\bd$ is identified with a small subgroup of $GL(n,\C)$ under
the group homomorphism given in the equation~(\ref{map_action}) of \ref{reduce_free}.
\end{enumerate}

By abuse of language we often say the space $X(\bd;A)$ is written in a normalized
form when we mean the type $(\bd;A)$ is normalized.
\end{defi}

\begin{prop}\label{prop-stab}
The space $X(\bd;A)$ is written in a normalized form if and only if the stabilizer
subgroup of $P$ is trivial for all~$P \in \C^n$ with exactly $n-1$ coordinates
different from zero.

In the cyclic case the stabilizer of a point as above (with exactly $\,n-1$
coordinates different from zero) has order $\gcd(d, a_1, \ldots, \widehat{a}_i,
\ldots, a_n)$.
\end{prop}

Using Lemma~\ref{movs-dA} is possible
to convert general types~$(\bd;A)$ into their normalized form.
Theorem~\ref{th_Prill} allows one to decide whether two quotient spaces are
isomorphic. In particular one can use this result to compute the singular points
of the space $X(\bd;A)$ and Proposition~\ref{prop-stab} to decide if a given cyclic
type is normalized.

In Example~\ref{Ex_quo_dim1} we explained the normalization process
in dimension one. In the following examples, we discuss the previous normalization process in
dimensions two and three separately.

\begin{ex}\label{X2} Following Lemma~\ref{upper_triangular}, all
quotient spaces for $n=2$ are cyclic. The space $X(d;a,b)$ is written in a
normalized form if and only if $\gcd(d,a) = \gcd(d,b) = 1$. If this is not the
case, one uses the isomorphism\footnote{Recall the notation $(i_1,\ldots,i_k) =
\gcd(i_1,\ldots,i_k)$ in case of complicated or long formulas.} (assuming
$\gcd(d,a,b)=1$)
$$
\begin{array}{rcl}
X(d; a,b)  & \longrightarrow & X \left( \frac{d}{(d,a)(d,b)}; \frac{a}{(d,a)},
\frac{b}{(d,b)} \right), \\[0.3cm]
\big[ (x,y) \big] & \mapsto & \big[ (x^{(d,b)},y^{(d,a)}) \big]
\end{array}
$$
to convert it into a normalized one, see also Lemma~\ref{movs-dA}.
\end{ex}

\begin{ex}\label{Ex_dim3} The quotient space $X(d;a,b,c)$ is
written in a normalized form if and only if $\gcd(d,a,b) = \gcd(d,a,c) =
\gcd(d,b,c) = 1$. As above, isomorphisms of the form $[(x,y,z)] \mapsto
[(x,y,z^k)]$ can be used to convert types $(d;a,b,c)$ into their normalized
form.
\end{ex}

In~\cite{Fujiki74} Fujiki computes resolutions of these cyclic quotient
singularities and also studies, among others, the previous properties of these
spaces.

\section{Weighted Projective Spaces}\label{weighted_projective_space}

Weighted projective spaces are main examples of $V$-varieties.
The main reference that has been used in this section is \cite{Dolgachev82}.
Here we concentrate our attention on describing the analytic structure
and singularities.

Let $\w=:(q_0,\ldots,q_n)$ be a weight vector, that is, a finite set of coprime positive
integers. 
There is a natural action of the multiplicative group $\C^{*}$ on
$\C^{n+1}\setminus\{0\}$ given by
$$
  (x_0,\ldots,x_n) \longmapsto (t^{q_0} x_0,\ldots,t^{q_n} x_n).
$$
The set of orbits $\frac{\C^{n+1}\setminus\{0\}}{\C^{*}}$ 
under this action is denoted by $\bP^n_\w$ (or $\bP^n(\w)$ in case of complicated weight vectors) 
and is called the {\em weighted projective space} of type $\w$. 
The class of a nonzero element $(x_0,\ldots,x_n)\in \C^{n+1}$ 
is denoted by $[x_0:\ldots:x_n]_\w$ and the weight vector is omitted depending on the context.
When $(q_0,\ldots,q_n)=(1,\ldots,1)$ one obtains the usual projective space
and the weight vector is always omitted. For $\mathbf{x}\in\C^{n+1}\setminus\{0\}$,
the closure of $[\mathbf{x}]_\w$ in $\C^{n+1}$ is obtained by adding the origin and it
is an algebraic curve.

\begin{nothing}\label{analytic_struc_Pkw}\textbf{Analytic structure.} As in the
classical case, weighted projective spaces can be endowed with an analytic
structure. However, in general they contain cyclic quotient singularities. To
understand this structure, consider the decomposition
$
\bP^n_\w = U_0 \cup \cdots \cup U_n,
$
where $U_i$ is the open set consisting of all elements $[x_0:\ldots:x_n]_\w$
with $x_i\neq 0$. The map
$$
  \widetilde{\psi}_0: \C^n \longrightarrow U_0,\quad
\widetilde{\psi}_0(x_1,\cdots,x_n):= [1:x_1:\ldots:x_n]_\w
$$
is clearly a surjective analytic map but it is not a chart since injectivity
fails. In fact, $[1:x_1:\ldots:x_n]_\w = [1:x'_1:\ldots,x'_n]_\w$
if and only if there exists $\xi\in\mu_{q_0}$ such that $x'_i = \xi^{q_i} x_i$
for all $i=1,\ldots,n$. Hence the map above induces the isomorphism
\begin{eqnarray*}
\psi_0\ :\ X(q_0;\, q_1,\ldots,q_n) & \longrightarrow & \ U_0,\\
\,[(x_1,\ldots,x_n)]\ & \mapsto & [1:x_1:\ldots:x_n]_\w.
\end{eqnarray*}
Analogously, $X(q_i;\,q_0,\ldots,\widehat{q}_i,\ldots,q_n) \cong U_i$
under the obvious analytic map. Since the \emph{changes of charts} are analytic,
$\bP^n_\w$~is an analytic space with
cyclic quotient singularities as claimed.
\end{nothing}

\begin{nothing}\textbf{Simplifying the weights.} For different weight vectors $\w$ and
$\w'$ the corresponding spaces
$\bP^n_\w$ and $\bP^n_{\w'}$ can be isomorphic. Consider
\begin{align*}
d_i & := \gcd (q_0,\ldots,\widehat{q}_i,\ldots,q_n),\\
e_i & := \lcm (d_0,\ldots,\widehat{d}_i,\ldots,d_n).
\end{align*}
Note that $e_i|q_i$,\, $\gcd(d_i,d_j)=1$ for $i\neq j$, and $\gcd(e_i,d_i)=1$.
\end{nothing}

\begin{prop}\label{propPw}
Using the notation above, the following map is an isomorphism:
$$
\begin{array}{rcl}
\bP^n \big(q_0,\ldots,q_n\big) & \longrightarrow & \bP^n
\displaystyle\big(\textstyle\frac{q_0}{e_0},\ldots,\frac{q_n}{e_n}\big),
\\[0.2cm]
\,[x_0:\ldots:x_n] & \mapsto &
\big[\,x_0^{d_0}:\ldots:x_n^{d_n}\,\big].
\end{array}
$$
\end{prop}

\begin{proof}
Since $\gcd(q_i,d_i)=1$ we have $e_i = d_0\cdot \ldots\cdot
\widehat{d}_i
\cdot \ldots\cdot d_n$.
 Now from Lemma~\ref{movs-dA} one has the following sequence
of isomorphisms
of analytic spaces:
\begin{equation*}
\begin{split}
& X(q_0;\, q_1, \ldots,q_n)\stackrel{\text{id}}{=}
\textstyle X(q_0;\,
\frac{q_1}{d_0},\frac{q_2}{d_0},\ldots,\frac{q_n}{d_0})\stackrel{1\text{st}}{
\cong}
X(\frac{q_0}{d_1};\, \frac{q_1}{d_0},\frac{q_2}{d_0 d_1}, \ldots,\frac{q_n}{d_0
d_1}) \\[0.10cm]
& \textstyle \stackrel{2\text{nd}}{\cong} X(\frac{q_0}{d_1 d_2};\,
\frac{q_1}{d_0 d_2},\frac{q_2}{d_0 d_1},
\frac{q_3}{d_0 d_1 d_2}, \ldots,\frac{q_n}{d_0 d_1
d_2})\stackrel{\text{3rd}}{\cong} \cdots
\stackrel{n\text{th}}{\cong} X(\frac{q_0}{e_0};\,
\frac{q_1}{e_1},\ldots,\frac{q_n}{e_n})
\end{split}
\end{equation*}
Observe that in the $i$-th step, we divide the corresponding weight vector by
$d_i$ except the $i$-th coordinate and hence the associated map is
$$[(x_1,\ldots,x_i,\ldots,x_n)] \longmapsto
[(x_1,\ldots,x_i^{d_i},\ldots,x_n)].$$
Therefore
$
[1:x_1:\ldots:x_n]_\w \longmapsto [1:x_1^{d_1}:\ldots:x_n^{d_n}]_{\w'}
$
is an isomorphism by composition. Analogously one proceeds with the other
charts.
\end{proof}

\begin{remark}
Note that, due to the preceding Proposition, one can always assume the weight
vector satisfies $\gcd(q_0,\ldots,\widehat{q}_i,\ldots,q_n)=1$, for
$i=0,\ldots,n$. In particular, $\bP^1{(q_0,q_1)} \cong \bP^1$ and for $n=2$ we can
take $(q_0,q_1,q_2)$ relatively prime numbers. In higher dimension the situation
is a bit more complicated.
\end{remark}

\section{Weighted Blow-ups and Embedded $\Q$-Resolutions}\label{sec-wbr}

Classically an embedded resolution of $\{f=0\} \subset \C^n$ is a proper map
$\pi: X \to (\C^n,0)$ from a smooth variety $X$ satisfying, among other
conditions, that $\pi^{-1}(\{f=0\})$ is a normal crossing divisor. To weaken the
condition on the preimage of the singularity we allow the new ambient space $X$
to contain abelian quotient singularities and the divisor $\pi^{-1}(\{f=0\})$ to
have \emph{normal crossings} over this kind of varieties. This notion of normal
crossing divisor on $V$-manifolds was first introduced by Steenbrink
in~\cite{Steenbrink77}.

\begin{defi}
Let $X$ be a $V$-manifold with abelian quotient singularities. A hypersurface
$D$ on $X$ is said to be with {\em $\mathbb{Q}$-normal crossings} if it is
locally isomorphic to the quotient of a union of coordinate hyperplanes under a group
action of type $(\bd;A)$. That is, given $x \in X$, there is an isomorphism of
germs $(X,x) \simeq (X(\bd;A), [0])$ such that $(D,x) \subset (X,x)$ is identified
under this morphism with a germ of the form
$$
\big( \{ [\mathbf{x}] \in X(\bd;A) \mid x_1^{m_1} \cdot\ldots\cdot x_k^{m_k} = 0 \},
[(0,\ldots,0)] \big).
$$
\end{defi}

Let $M = \C^{n+1} / \mu_\bd$ be an abelian quotient space not necessarily cyclic
or written in normalized form.
Consider $H \subset M$ an analytic subvariety of codimension one.

\begin{defi}\label{Qresolution}
An {\em embedded $\Q$-resolution} of $(H,0) \subset (M,0)$ is a proper analytic
map~$\pi: X \to (M,0)$ such that:
\begin{enumerate}
\item $X$ is a $V$-manifold with abelian quotient singularities.
\item $\pi$ is an isomorphism over $X\setminus \pi^{-1}(\Sing(H))$.
\item $\pi^{-1}(H)$ is a hypersurface with $\mathbb{Q}$-normal crossings on $X$.
\end{enumerate}
\end{defi}

\begin{remark}
Let $f:(M,0) \to (\C,0)$ be a non-constant analytic function germ. Consider
$(H,0)$ the hypersurface defined by $f$ on $(M,0)$. Let $\pi: X \to (M,0)$ be an
embedded $\Q$-resolution of $(H,0) \subset (M,0)$. Then $\pi^{-1}(H) = (f\circ
\pi)^{-1}(0)$ is locally given by a function of the form
$
x_1^{m_1}  \cdot\ldots\cdot x_k^{m_k} : X(\bd;A) \rightarrow \C.
$
for some type $(\bd;A)$.
\end{remark}

\begin{nothing}\label{blowup_dimN_smooth}\textbf{Classical blow-up of $\C^{n+1}$.}
Using
multi-index notation
we consider
$$
\widehat{\C}^{n+1} := \big\{ (\mathbf{x},[\mathbf{u}]) \in \C^{n+1} \times \bP^n \mid
\mathbf{x}\in \overline{[\mathbf{u}]} \big\}.
$$

Then $\pi : \widehat{\C}^{n+1} \to \C^{n+1}$ is an isomorphism over
$\widehat{\C}^{n+1} \setminus \pi^{-1}(0)$. The {\em exceptional divisor} $E:=
\pi^{-1}(0)$ is identified with $\bP^{n}$. The space
$\widehat{\C}^{n+1} = U_0 \cup \cdots \cup U_n$ can be covered with $n+1$
charts each of them isomorphic to $\C^{n+1}$. For instance, the following map
defines an isomorphism:
\begin{eqnarray*}
\C^{n+1} & \longrightarrow & U_0 = \{ u_0 \neq 0 \} \ \subset \
\widehat{\C}^{n+1},\\
\mathbf{x} \ & \mapsto & \big( (x_0, x_0 x_1, \ldots, x_0 x_n),[1: x_1 : \ldots : x_n] \big).
\end{eqnarray*}
\end{nothing}

\begin{nothing}\label{weighted_blowup_dimN_smooth}\textbf{Weighted
$(p_0,\ldots,p_n)$-blow-up of $\C^{n+1}$.} 
Let~$\w = (p_0, \ldots, p_n)$ be a weight
vector. As above, consider the space
$$
\widehat{\C}^{n+1}(\w) := \big\{ (\mathbf{x},[\mathbf{u}]_{\w}) \in \C^{n+1} \times
\bP^n_{\w} \mid \mathbf{x} \in \overline{[\mathbf{u}]}_{\w}
\big\}.
$$
Here the condition about the closure means that
$
\exists t \in \C, \quad x_i = t^{p_i} u_i, \quad i=0,\ldots,n.
$
Then the natural projection $\pi : \widehat{\C}^{n+1}(\w) \to \C^{n+1}$ is an
isomorphism over $\widehat{\C}^{n+1}(\w) \setminus \pi^{-1}(0)$ and the
exceptional divisor $E:= \pi^{-1}(0)$ is identified with~$\bP^n_{\w}$. 
Again the space 
$
\widehat{\C}^{n+1}(\w) = U_0 \cup \cdots \cup
U_n
$ 
can be covered with $n+1$ charts. However, $\varphi_0 : \C^{n+1} \to U_0$
given by
\begin{eqnarray*}
\C^{n+1} & \stackrel{\varphi_0}{\longrightarrow} & U_0 = \{ u_0 \neq 0 \} \
\subset \ \widehat{\C}^{n+1}(\w),\\
\mathbf{x} \ & \mapsto & \big( (x_0^{p_0}, x_0^{p_1} x_1, \ldots, x_0^{p_n} x_n),
[1: x_1 : \ldots : x_n]_{\w} \big),
\end{eqnarray*}
is surjective but not injective. In fact 
$\varphi_0(\mathbf{x}) = \varphi_0({\bf y})$ if and only if
$$
\exists \xi \in \mu_{p_0}: 
\begin{cases}
 y_0 = \xi^{-1} x_0, \\ y_i = \xi^{p_i}
x_i, & i = 1, \ldots,  n. 
\end{cases}
$$
Hence the previous map $\varphi_0$ induces an isomorphism
$$
X(p_0; -1, p_1, \ldots, p_n)  \longrightarrow U_0.
$$
Note that these charts are compatible with the ones described
in~\ref{analytic_struc_Pkw} for the weighted projective space. In $U_0$ the
exceptional divisor is $\{x_0=0\}$ and the first chart of $\bP^n_{\w}$ is the
quotient space $X(p_0;\, p_1, \ldots, p_n)$.
\end{nothing}

\begin{nothing}\textbf{Weighted $\w$-blow-up of $X(\bd; A)$.}
The action $\mu_\bd$ on $\C^{n+1}$ extends
naturally to an action on~$\widehat{\C}^{n+1}(\w)$ as follows,
$$
\bxi_\bd \cdot \big( \mathbf{x}, [\mathbf{u}]_{\w} \big) \stackrel{\mu_d}{\longmapsto}
\left( ( \bxi_\bd^{\ba_0} x_0,\ldots, \bxi_\bd^{\ba_n} x_n), [\bxi_\bd^{\ba_0} u_0: \ldots:
\bxi_\bd^{\ba_n} u_n]_{\w} \right).
$$

Let $\widehat{X(\bd;A)}(\w) := \widehat{\C}^{n+1}(\w) \big/ \mu_\bd$ denote the
quotient space under this action. Then the induced projection
$$
\pi: \widehat{X(\bd;A)}(\w) \longrightarrow X(\bd;A), \quad \big[ (\mathbf{x}, [{\bf
u}]_{\w}) \big]_{(\bd;A)} \mapsto [\mathbf{x}]_{(\bd;A)}
$$
is an isomorphism over $\widehat{X(\bd;A)}(\w) \setminus \pi^{-1}([0])$ and the
exceptional divisor $E:=\pi^{-1}([0])$ is identified with $\bP^n_{\w}/ \mu_\bd$.

The action $\mu_\bd$ above respects the charts of $\widehat{\C}^{n+1}(\w)$ so
that the new ambient space can be covered as
$$
\widehat{X(\bd;A)}(\w) = \widehat{U}_0 \cup \cdots \cup \widehat{U}_n,
$$
where $\widehat{U}_i := U_i / \mu_\bd = \{ u_i \neq 0 \}$.

Let us study for instance the first chart.
By using $\varphi_0$ one identifies~$U_0$ with 
$$
X(p_0;-1,p_1,\ldots,p_n).
$$ 
Let us consider the space $X(p_0 \bd;(\ba_0\vert p_0 \ba_1 - p_1 \ba_0\vert\dots\vert p_0 \ba_n - p_n \ba_0))$.
Since the following diagram commutes:
$$
\begin{matrix}
(x_0,x_1,\dots,x_n)&\mapsto &
(\bxi_{p_0 \bd}^{\ba_0} x_0, \bxi_{p_0 \bd}^{ p_0 \ba_1 - p_1 \ba_0} x_1,\dots,\bxi_{p_0 \bd}^{ p_0 \ba_n - p_n \ba_0} x_n)\\
\downarrow&&\downarrow\\
(x_0^{p_0},x_0^{p_1} x_1,\dots,x_0^{p_n} x_n)&\mapsto &
(\bxi_{p_0 \bd}^{p_0 \ba_0} x_0^{p_0},\bxi_{p_0 \bd}^{p_0 \ba_1} x_0^{p_1} x_1,\dots,\bxi_{p_0 \bd}^{p_0 \ba_n}  x_0^{p_n} x_n)
\end{matrix}
$$
where the vertical arrows are the chart morphisms, the upper arrow corresponds to the action
defining $X(p_0 \bd;(\ba_0\vert p_0 \ba_1 - p_1 \ba_0\vert\dots\vert p_0 \ba_n - p_n \ba_0))$
and the lower arrow to the one defining $X(p_0 \bd;p_0 A)=X(\bd; A)$.
Since the chart is an isomorphism for  $X(p_0;-1,p_1,\ldots,p_n)$
we conclude that
$$
X
\left(
\begin{pmatrix}
p_0\\
\hrulefill\\
p_0 \bd 
\end{pmatrix}
;
\left(
\begin{matrix}
-1\\
\hrulefill\\
\ba_0
\end{matrix}
\Biggm\vert
\begin{matrix}
p_1\\
\hrulefill\\
p_0 \ba_1 - p_1 \ba_0
\end{matrix}
\Biggm\vert
\dots
\Biggm\vert
\begin{matrix}
p_n\\
\hrulefill\\
p_0 \ba_n - p_n \ba_0
\end{matrix}
\right)
\right)
$$
is isomorphic to $\widehat{U}_0$ and the isomorphism is defined by
$$
[\mathbf{x}] \ \stackrel{\widehat{\varphi}_0}{\longmapsto} \ \big[\big( (x_0^{p_0},
x_0^{p_1} x_1, \ldots, x_0^{p_n} x_n),[1: x_1 : \ldots :
x_n]_{\w} \big)\big].
$$
Analogously one proceeds for $i=1,\ldots,n$.
As for the the exceptional divisor $E = \pi^{-1}(0) = \bP^n_{\w} / \mu_\bd$, 
it is as usual covered as $\widehat{V}_0 \cup \cdots \cup
\widehat{V}_n$ so that these charts are compatible with the ones of
$\widehat{X(\bd;A)}(\w)$ in the sense that $\widehat{V}_i =
\widehat{U}_i|_{\{x_i = 0\}}$, $i = 0, \ldots, n$. Hence, for example,
$$
\widehat{V}_0  \cong  X \left(\begin{array}{c|cccccc}
p_0 & p_1 & \cdots & p_n  \\
p_0 \bd & p_0 \ba_1 - p_1 \ba_0 & \cdots & p_0 \ba_n -p_n \ba_0 
\end{array}\right).
$$
\end{nothing}

\begin{ex}\label{blow-up2-sing-ab}
We study the weighted blow-up
$$
\pi := \pi_{(d;a,b),\w}: \, \widehat{X(d;a,b)}_{\w}
\longrightarrow X(d;a,b)
$$ 
of the origin of a normalized space $X(d;a,b)$
with respect to $\w = (p,q)$. The new space is covered as
$$
\widehat{U}_1 \cup \widehat{U}_2 = X \left( \begin{array}{c|cc} p & -1 & q \\ pd
& a & pb - qa \end{array} \right) \cup X \left( \begin{array}{c|cc} q & p & -1
\\ qd & qa-pb & b \end{array} \right)
$$
and the charts are given by 
$$
\begin{array}{c|c}
\text{First chart} & X \left( \begin{array}{c|cc} p & -1 & q \\ pd & a & pb - qa
\end{array} \right) \ \longrightarrow \ \widehat{U}_1, \\[0.5cm] & \,\big[ (x,y)
\big] \mapsto \big[ ((x^p,x^q y),[1:y]_{\w}) \big]_{(d;a,b)}. \\
\multicolumn{2}{c}{} \\
\text{Second chart} & X \left( \begin{array}{c|cc} q & p & -1 \\ qd & qa-pb & b
\end{array} \right) \ \longrightarrow \ \widehat{U}_2, \\[0.5cm]
& \big[ (x,y) \big] \mapsto \big[ ((x y^p, y^q),[x:1]_{\w}) \big]_{(d;a,b)}.
\end{array}
$$
The exceptional divisor $E = \pi_{(d;a,b),\w}^{-1}(0)$ is identified with the
quotient space $\bP^1_{\w}(d;a,b) := \bP^1_{\w}/\mu_{d}$ which is isomorphic to
$\bP^1$ under the map
$$
\begin{array}{rcl}
\bP^1_{\w}(d;a,b) & \longrightarrow & \bP^1, \\[0.1cm]
\,[x:y]_{\w} & \mapsto & [x^{dq/e}: y^{dp/e}],
\end{array}
$$
where $e: = \gcd(d,pb-qa)$. Again the singular points are cyclic and
correspond to the origins. 
To normalize these quotient spaces, note
that $e =\gcd(d,pb-qa)= \gcd(d, -q+\beta p b) =\gcd(pd, -q+\beta p b)=\gcd(qd,p-q a\mu)$, where
$\beta a\equiv \mu b\equiv 1\mod d$.

Then another expression for the two charts are given below.
\begin{equation}\label{blowup_normalized_dim2}
\begin{array}{c|c}
\text{First chart} & X \left( \displaystyle\frac{pd}{e}; 1, \frac{-q+\beta p
b}{e} \right)  \ \longrightarrow \ \widehat{U}_1, \\[0.5cm] & \,\big[ (x^e,y)
\big] \mapsto \big[ ((x^p,x^q y),[1:y]_{\w}) \big]_{(d;a,b)}. \\
\multicolumn{2}{c}{} \\
\text{Second chart} & X \left( \displaystyle\frac{qd}{e}; \frac{-p+\mu qa}{e}, 1
\right) \ \longrightarrow \ \widehat{U}_2, \\[0.5cm] & \hspace{0.15cm} \big[
(x,y^e) \big] \mapsto \big[ ((x y^p, y^q),[x:1]_{\w}) \big]_{(d;a,b)}.
\end{array}
\end{equation}

Both quotient spaces are now written in their normalized
form.
\end{ex}

\begin{ex}
Let $\pi := \pi_{\w}: \widehat{\C}^3_{\w} \to \C^3$ the weighted blow-up at the origin with respect to $\w=(p,q,r)$,
$\gcd\w=1$. The new space is covered as
$$
\widehat{\C}^3_{\w} = U_1 \cup U_2 \cup U_3 = X(p;-1,q,r) \cup X(q;p,-1,r) \cup X(r;p,q,-1),
$$
and the charts are given by
\begin{equation}\label{charts_dim3}
\begin{array}{cc}
X(p;-1,q,r) \longrightarrow U_1: & [(x,y,z)] \mapsto ((x^p, x^q y, x^r z),[1:y:z]_{\w}), \\[0.25cm]
X(q;p,-1,r) \longrightarrow U_2: & [(x,y,z)] \mapsto ((x y^p,y^q,y^r z),[x:1:z]_{\w}), \\[0.25cm]
X(r;p,q,-1) \longrightarrow U_3: & [(x,y,z)] \mapsto ((x z^p, y z^q, z^r),[x:y:1]_{\w}).
\end{array}
\end{equation}

In general $\widehat{\C}^3_{\w}$ has three lines of (cyclic quotient) singular
points located at the three axes of the exceptional divisor $\pi^{-1}_{\w}(0) \simeq \bP^2_{\w}$.
Namely, a generic point in $x=0$ is a cyclic point of type
$\C\times X(\gcd(q,r);p,-1)$.

Note that although the quotient spaces are written in their normalized form,
the exceptional divisor can be simplified:
$$
\begin{array}{rcl}
\bP^2(p,q,r) & \longrightarrow & \bP^2 \displaystyle\left(\frac{p}{(p,r)\cdot
(p,q)},\frac{q}{(q,p)\cdot (q,r)},
\frac{r}{(r,p)\cdot (r,q)}\right),\\[0.5cm]
\displaystyle \,[x:y:z] & \mapsto & [x^{\gcd(q,r)}:y^{\gcd(p,r)}:z^{\gcd(p,q)}].
\end{array} 
$$
Using just a weighted blow-up of this kind, one can find an embedded
$\Q$-resolution for Brieskorn-Pham surfaces singularities, i.e.~$x^a + y^b + z^c
= 0$. For simplicity, let us assume $a,b,c$ are pairwise coprime and consider the weight
$\w:=(p,q,r)$ where $p:=b c$, $q:=a c$ and $r=a b$.
The blow-up $\pi_\w$ gives a $\Q$-resolution of this singularity.
The exceptional divisor $E$ is isomorphic to $\bP^2_\w\cong\bP^2$.
The intersection of the strict transform with $E$ is a generic line in $\bP^2$.
Note that the space $\C^3_\w$ has three singular points of type
$(b c;-1, a c,a b)$, $(a c;b c,-1,a b)$ and $(a b;b c,a c,-1)$ and
three strata (isomorphic to $\C^*$) with transversal singularities of type 
$X(a;b c,-1)$, $X(b;a c,-1)$ and $X(c;a b,-1)$.
\end{ex}

\section{Cartier and Weil Divisors on $V$-Manifolds: $\mathbb{Q}$-Divisors}\label{sec-cw}

The aim of this section is to show that when $X$ is a $V$-manifold there is an
isomorphism of $\mathbb{Q}$-vector spaces between Cartier and Weil divisors, see
Theorem~\ref{cartier_weil_XdA} below. It is explained
in~\ref{how_to_write_summarize} how to write explicitly a $\mathbb{Q}$-Weil
divisor as a $\mathbb{Q}$-Cartier divisor. Also, the case of the exceptional
divisor of a weighted blow-up is treated in Example~\ref{how_to_write_example}.

\subsection{Divisors on complex analytic varieties}
\mbox{}

Let $X$ be an irreducible complex analytic variety. As usual, consider~$\cO_X$
the structure sheaf of $X$ and $\mathcal{K}_X$ the sheaf of total quotient rings
of $\cO_X$. Denote by $\mathcal{K}_X^{*}$ the (multiplicative) sheaf of
invertible elements in $\mathcal{K}_X$. Similarly $\cO_X^{*}$ is the sheaf of
invertible elements in~$\cO_X$. 
(resp.~$\mathcal{K}_X$) is the sheaf of holomorphic (resp.~meromorphic)
functions and $\cO_X^{*}$ (resp.~$\mathcal{K}_X^{*}$) is the sheaf of nowhere
vanishing holomorphic (resp.~non-zero meromorphic) functions.

\begin{remark}
By {\em complex analytic variety} we mean a reduced complex space. A {\em
subvariety} $V$ of $X$ is a reduced closed complex subspace of $X$, or
equivalently, an analytic set in~$X$, cf.~\cite{GR84}. An irreducible subvariety
$V$ corresponds to a prime ideal in the ring of sections of any local complex
model space meeting $V$.
\end{remark}

\begin{defi}
A {\em Cartier divisor} on $X$ is a global section of the sheaf
$\mathcal{K}_X^{*}/\cO_X^{*}$, that is, an element in
$\Gamma(X,\mathcal{K}_X^{*}/\cO_X^{*}) = H^0(X,\mathcal{K}_X^{*}/\cO_X^{*}) $.
Any Cartier divisor can be represented by giving an open covering $\{U_i\}_{i\in
I}$ of $X$ and, for all $i\in I$, an element $f_i \in \Gamma (U_i,
\mathcal{K}_X^{*})$ such that
$$
\frac{f_i}{f_j} \in \Gamma(U_i \cap U_j, \cO_X^{*}), \quad \forall i,j \in I.
$$
Two systems $\{(U_i,f_i)\}_{i\in I}$, $\{(V_j,g_j)\}_{j\in J}$ represent the
same Cartier divisor if and only if on $U_i \cap V_j$, $f_i$ and $g_j$ differ by
a multiplicative factor in $\cO_X(U_i\cap V_j)^{*}$.
The abelian group of Cartier divisors on $X$ is denoted by $\CaDiv(X)$. If $D :=
\{(U_i,f_i)\}_{i\in I}$ and $E := \{(V_j,g_j)\}_{j\in J}$ then $D + E = \{(U_i
\cap V_j, f_i g_j)\}_{i\in I, j\in J}$.
\end{defi}

The functions $f_i$ above are called {\em local equations} of the divisor on
$U_i$. A Cartier divisor on $X$ is {\em effective} if it can be represented by
$\{(U_i,f_i)\}_i$ with all local equations $f_i \in \Gamma(U_i,\cO_{X})$.

Any global section~$f \in \Gamma(X,\mathcal{K}^{*}_X)$ determines a {\em
principal} Cartier divisor $(f)_X := \{(X,f)\}$ by taking all local equations
equal to $f$. That is, a Cartier divisor is principal if it is in the image of
the natural map $\Gamma(X,\mathcal{K}_X^{*}) \to
\Gamma(X,\mathcal{K}_X^{*}/\cO_X^{*})$. Two Cartier divisors $D$ and $E$ are {\em
linearly equivalent}, denoted by $D\sim E$, if they differ by a principal
divisor. The {\em Picard group} $\Pic(X)$ denotes the group of linear
equivalence classes of Cartier divisors.

The {\em support} of a Cartier divisor $D$, denoted by $\Supp(D)$ or $|D|$, is the subset of $X$ consisting of all points $x$ such that a local equation for $D$ is not in $\cO_{X,x}^{*}$. The support of $D$ is a closed subset of $X$.

\begin{defi}\label{defi_Weil}
A {\em Weil divisor} on $X$ is a locally finite linear combination with integral
coefficients of irreducible subvarieties of codimension one. The abelian group
of Weil divisors on $X$ is denoted by $\WeDiv(X)$. If all coefficients appearing
in the sum are non-negative, the Weil divisor is called {\em effective}.
\end{defi}

\begin{remark}
In the algebraic category 
the locally finite sum of Definition~\ref{defi_Weil} is
automatically finite. Therefore $\WeDiv(X)$ is the free abelian group on the
codimension one irreducible algebraic subvarieties of $X$. Similar
considerations hold if $X$ is a compact analytic variety.
\end{remark}

Given a Cartier divisor there is a Weil divisor associated with it. To see this,
the notion of order of a divisor along an irreducible subvariety of codimension
one is needed.

\begin{nothing}\label{order_function}\textbf{Order function.} Let $V \subset X$ be an
irreducible subvariety of codimension one. It corresponds to a prime ideal in
the ring of sections of any local complex model space meeting~$V$. The {\em
local ring of $X$ along $V$}, denoted by $\cO_{X,V}$, is the localization of such
ring of sections at the corresponding prime ideal; it is a one-dimensional local
domain.

For a given $f \in \cO_{X,V}$ define $\ord_V(f)$ to be
$$
  \ord_V(f) := \length_{\cO_{X,V}} \left( \frac{\cO_{X,V}}{\langle f \rangle}
\right).
$$
This determines a well-defined group homomorphism $\ord_V:
\Gamma(X,\mathcal{K}_X^{*}) \to \Z$ satisfying, for a given $f \in
\Gamma(X,\mathcal{K}_X^{*})$, the following property: 
$$
  \forall x \in X, \ \exists U_x \subset X
\text{ open neighborhood of }x \, \mid \, \# \{ \ord_V(f) \neq 0
\mid V \cap U_x \neq \emptyset \} < +\infty.
$$

The previous length, $X$ being a complex analytic variety of dimension $n \geq
2$, can be computed as follows. Choose $x \in V$ such that $x$ is smooth in $X$
and $(V,x)$ defines an irreducible germ. This germ is the zero set of an
irreducible $g \in \cO_{X,x}$. Then
$
  \ord_V(f) = \ord_{V,x}(f),
$
where $\ord_{V,x}(f)$ is the classical order of a meromorphic function at a
smooth point with respect to an irreducible subvariety of codimension one; it is
known to be given by the equality $f = g^{\ord} \cdot h \in \cO_{X,x}$ where $h
\nmid g$.
The same applies if $X$ is $1$-dimensional and smooth.
\end{nothing}

\begin{remark}
The order $\ord_{V,x}(f)$ does not depend on the defining equation $g$, as long
as we choose $g$ irreducible. In fact, two irreducible $g, g'\in \cO_{X,x}$ with
$V(g) = V(g')$ only differ by a unit in $\cO_{X,x}$. Moreover, $\ord_{V,x}(f)$
does not depend on $x$, since the set $V_{\text{red}}$ of regular points is
connected if $V$ is irreducible.
\end{remark}

Now if $D$ is a Cartier divisor on $X$, one writes $\ord_V (D) = \ord_V(f_i)$
where $f_i$ is a local equation of $D$ on any open set $U_i$ with $U_i \cap V
\neq \emptyset$. This is well defined since $f_i$ is uniquely determined up to
multiplication by units and the order function is a homomorphism. Define the
{\em associated Weil divisor} of a Cartier divisor~$D$ by  
$$
\begin{array}{cccl}
T_X: & \CaDiv(X) & \longrightarrow & \WeDiv(X), \\[0.15cm]
& D & \longmapsto & \displaystyle \sum_{V \subset X} \ord_V(D) \cdot [V],
\end{array}
$$
where the sum is taken over all codimension one irreducible subvarieties $V$ of
$X$. The previous sum is locally finite, i.e.~for any $x\in X$ there exists an
open neighborhood $U$ such that the set $\{ \ord_V(D) \neq 0 \mid V \cap U \neq
\emptyset \}$ is finite. By the additivity of the order function, the mapping
$T_X$ is a homomorphism of abelian groups.

A Weil divisor is {\em principal} if it is the image of a principal Cartier
divisor under $T_X$; they form a subgroup of $\WeDiv(X)$. If $\Cl(X)$ denotes
the quotient group of their equivalence classes, then $T_X$ induces a morphism
$\Pic(X) \rightarrow \Cl(X)$.

These two homomorphisms ($T_X$ and the induced one) are in general neither
injective nor surjective. In this sense one has the following result.

\begin{theo}{\em (cf.~\cite[21.6]{GD67}).}\label{cartier_weil_prop}~If $X$ is
normal (resp.~locally factorial) then the previous maps $\CaDiv(X) \to
\WeDiv(X)$ and $\Pic(X) \to \Cl(X)$ are injective (resp.~bijective). The image
of the first map is the subgroup of locally principal\footnote{A Weil divisor
$D$ on $X$ is said to be {\em locally principal} if $X$ can be covered by open
sets $U$ such that $D|_U$ is principal for each $U$.} Weil divisors. 
\end{theo}

\begin{remark}\label{remark_cartier_weil_prop}
Locally factorial essentially means that every local ring $\cO_{X,x}$ is a unique
factorization domain. In particular, every smooth analytic variety is locally
factorial. In such a case, Cartier and Weil divisors are identified and denoted
by $\Div(X) := \CaDiv(X) = \WeDiv(X)$. Their equivalence classes coincide under
this identification and we often write $\Pic(X) = \Cl(X)$.
\end{remark}

\begin{ex}\label{ex_non-cartier_weil_divisor}
Let $X$ be the surface in $\C^3$ defined by the equation $z^2=xy$. The line
$V=\{x=z=0\}$ defines a Weil divisor which is not a Cartier divisor. In this
case $\Pic(X)=0$ and $\Cl(X) = \Z/(2)$. Note that $X$ is normal but not locally
factorial. 
However, the associated Weil divisor of $\{(X,x)\}$ is
$$
  T_X \big( \{(X,x) \} \big) = \sum_{\begin{subarray}{c} Z\subset X,\,\,
\text{irred}\\ \codim(Z)=1 \end{subarray}} \hspace{-0.1cm} \ord_Z(x) \cdot [Z] =
2 [V].
$$
Thus $[V]$ is principal as an element in $\WeDiv(X) \otimes_{\Z} \mathbb{Q}$ and
corresponds to the $\mathbb{Q}$-Cartier divisor $\frac{1}{2} \{(X,x)\}$. 

This fact can be interpreted as follows. First note that identifying our
surface~$X$ with $X(2;1,1)$ under $[(x,y)] \mapsto (x^2,y^2,xy),$ the previous
Weil divisor corresponds to $D=\{x=0\}$. Although $f=x$ defines a zero set on
$X(2;1,1)$, it does not induce a function on the quotient space.
However, $x^2: X(2;1,1) \to \C$ is a well-defined function and gives rise to the
same zero set as $f$. Hence as $\mathbb{Q}$-Cartier divisors one writes $D =
\frac{1}{2} \{(X(2;1,1),x^2)\}.$
\end{ex}

\subsection{Divisors on $V$-manifolds}

\mbox{}

Example~\ref{ex_non-cartier_weil_divisor} above illustrates the general behavior
of Cartier and Weil divisors on $V$-manifolds, namely Weil divisors are all
locally principal over $\mathbb{Q}$. To prove it we need some preliminaries.

\begin{lemma}\label{lemma_cartier_weil_XdA}
Let $B \subset \C^n$ be an open ball and let $G$ be a finite group acting on
$B$. Then one has $\Cl(B/G) \otimes_{\Z} \mathbb{Q} = 0$.
\end{lemma}

\begin{proof}
Let $V \subset B/G =:U$ be an irreducible subvariety of codimension one. We
shall prove that there exists $k \geq 1$ such that $k [V] \in \WeDiv(U)$ is
principal.

Consider the natural projection $\pi: B \to U$. Then $W = \pi^{-1}(V)$ gives
rise to an effective Weil divisor on the open ball $B$. 
There exists $f:B \to \C$ a holomorphic function such that $W = \{f=0\} \subset
B$. Thus,
$$
V = \pi(W) = \{ [\mathbf{x}] \mid \mathbf{x}\in B,\, f(\mathbf{x}) = 0 \} = \{ f=0 \}
\subset U.
$$

Moreover, the holomorphic function $f$ must satisfy the following property
\begin{equation}\label{property_for_quasi-function}
  \forall P \in U,\, \big[ f(P) = 0 \, \Longrightarrow \, f( \sigma \cdot P ) =
0, \forall \sigma \in G \big].
\end{equation}

Since $f$ does not induce an analytic function on $U$, although $V$ is given by
just one equation, it is not necessarily principal as an element in $\WeDiv(U)$,
see Example~\ref{ex_non-cartier_weil_divisor}. Now the main idea is to change
$f$ by another holomorphic function $F$ such that $V=\{F=0\}$ but now with $F
\in \Gamma(U,\cO_{U})$.

Let us consider $F = \prod_{\sigma \in G} f^{\sigma}$ where $f^{\sigma}(\mathbf{x})
= f(\sigma \cdot \mathbf{x})$; clearly it verifies the previous conditions. Then
$\{(U,F)\}$ is a principal Cartier divisor and its associated Weil divisor is 
$$
  T_U \big( \{(U,F)\} \big) = \sum_{\begin{subarray}{c} Z\subset U,\,\,
\text{irred}\\ \codim(Z)=1 \end{subarray}} \hspace{-0.1cm} \ord_{Z}(F) \cdot [Z]
= \ord_V(F) \cdot [V].
$$
Note that $\ord_Z(F) \neq 0$ implies $Z=V$, since $V$ is irreducible.
\end{proof}

\begin{remark}
Note that the proof of this result is based on an idea extracted
from~\cite[Ex.~1.7.6]{Fulton98}, where an explicit isomorphism between
$\WeDiv(B/G) \otimes_{\Z} \mathbb{Q}$ and $(\WeDiv(B) \otimes_{\Z}
\mathbb{Q})^G$ is given.
\end{remark}

\begin{theo}\label{cartier_weil_XdA}
Let $X$ be a $V$-manifold. Then the notion of Cartier and Weil divisor coincide
over $\mathbb{Q}$. More precisely, the linear map 
$$
  T_X \otimes 1: \CaDiv(X) \otimes_{\Z} \mathbb{Q} \longrightarrow \WeDiv(X)
\otimes_{\Z} \mathbb{Q} 
$$
is an isomorphism of $\mathbb{Q}$-vector spaces. In particular, for a given Weil
divisor $D$ on $X$ there always exists $k \in \Z$ such that $kD \in \CaDiv(X)$.
\end{theo}

\begin{proof}
The variety $X$ is normal and then Theorem~\ref{cartier_weil_prop} applies.
Therefore the linear map $T_X \otimes 1$ is injective and its image is the
$\mathbb{Q}$-vector space generated by the locally principal Weil divisors on
$X$.

Let $V \subset X$ be an irreducible subvariety of codimension one. Consider
$\{U_i\}_i$ an open covering of $X$ such that $U_i$ is analytically isomorphic
to $B_i/G_i$ where $B_i \subset \C^{n}$ is an open ball and $G_i$ is a finite
subgroup of $GL(n,\C)$. By Lemma~\ref{lemma_cartier_weil_XdA}, $\Cl(U_i) \otimes
\mathbb{Q} = 0$ for all $i$.

Thus $[V|_{U_i}]$ is principal as an element in $\WeDiv(U_i) \otimes_{\Z}
\mathbb{Q}$ which implies that $V$ is locally principal over $\mathbb{Q}$ and
hence belongs to the image of $T_X \otimes 1$.
\end{proof}

\subsection{Writing a Weil divisor as a $\mathbb{Q}$-Cartier divisor}
\mbox{}

Following the proofs of Lemma~\ref{lemma_cartier_weil_XdA} and
Theorem~\ref{cartier_weil_XdA}, every Weil divisor on~$X$ can locally be written
as $\mathbb{Q}$-Cartier divisor like
$$
[V|_U] = \frac{1}{\ord_V(F)} \{(U,F)\}
$$
where $F = \prod_{\sigma \in G} f^{\sigma}$ and $V \cap U =\{f=0\}$ with $f:B
\to \C$ being holomorphic on an open ball and
satisfying~(\ref{property_for_quasi-function}).

The rest of this section is devoted to calculate explicitly $\ord_V(F)$.
Firstly, it is shown in Proposition~\ref{prop_not_well-defined_holomorphic} that
$F$ is essentially a power of $f$, if the latter is chosen properly. Then,
$\ord_V(F)$ is computed in Proposition~\ref{computation_ord_V_F}.

\begin{prop}\label{prop_not_well-defined_holomorphic}
Let $f:B \to \C$ be a non-zero holomorphic function on an open ball $B \subset
\C^n$ such that the germ $f_x \in \cO_{B,x}$ is reduced for all $x \in B$. Let
$G$ be a finite subgroup of $GL(n,\C)$ acting on $B$. As above consider $F =
\prod_{\sigma \in G} f^{\sigma}$ where $f^{\sigma}(\mathbf{x}) = f(\sigma \cdot
\mathbf{x})$.
The following conditions are equivalent: 
\begin{enumerate}[\rm(1)]
\item\label{prop_not_well-defined_holomorphic1} $\forall P \in B, \, \big[ f(P)=0 \, \Longrightarrow \, f(\sigma \cdot P)
= 0, \, \forall \sigma \in G \big]$.

\item\label{prop_not_well-defined_holomorphic2} $\forall \sigma \in G$, $\exists h_{\sigma} \in \Gamma (B, \cO_{B}^{*})$
such that $f^{\sigma} = h_{\sigma} f$.

\item\label{prop_not_well-defined_holomorphic3} $\exists h \in \Gamma(B,\cO_B^{*})$ such that $F = h f^{|G|}$.

\item\label{prop_not_well-defined_holomorphic4} $\exists k \geq 1$, $\exists h \in \Gamma (B,\cO_{B}^{*})$ such that $h f^k
\in \Gamma(B/G,\cO_{B/G})$.
\end{enumerate}
\end{prop}

\begin{proof}
For $\eqref{prop_not_well-defined_holomorphic1} \Rightarrow \eqref{prop_not_well-defined_holomorphic2}$, 
first note that $f^{\sigma} \in IV(f)$ (the ideal of the zero locus of $f$). Now fix $x\in
B$. Since $f_x$ is reduced, there exists a holomorphic function $h$ on a small
enough open neighborhood of $x$ such that as germs $(f^\sigma)_x = h_x f_x$. The
order of the converging power series $(f^{\sigma})_x$ and $f_x$ are equal
because the action is linear. Thus $h_x$ is a unit in $\cO_{B,x}$. In particular
$\frac{f^{\sigma}}{f}$ is holomorphic and does not vanish at $x \in B$.

For $\eqref{prop_not_well-defined_holomorphic2} \Rightarrow \eqref{prop_not_well-defined_holomorphic3}$, 
consider $h = \prod_{\sigma \in G} h_{\sigma}$. Then
one has
$$
  F = \prod_{\sigma \in G} f^{\sigma} = \prod_{\sigma \in G} (h_{\sigma} f) =
\bigg( \prod_{\sigma \in G} h_{\sigma} \bigg) \cdot f^{|G|} = h f^{|G|}.
$$

For $\eqref{prop_not_well-defined_holomorphic3} \Rightarrow \eqref{prop_not_well-defined_holomorphic4}$, 
since $F:B/G \to \C$ is analytic,  take $k = |G|$.
Finally, note that $\forall P \in B$, $f(P) = 0 \Longleftrightarrow (hf^{k})(P)
= (hf^k)(\sigma \cdot P) = 0 \Longleftrightarrow f(\sigma \cdot P) = 0$. Hence
$\eqref{prop_not_well-defined_holomorphic4} \Rightarrow \eqref{prop_not_well-defined_holomorphic1}$ 
follows and the proof is complete.
\end{proof}

This example shows that the reduced condition in the statement of the previous
result is necessary.

\begin{ex}\label{bad_exam_1}
Let $f = (x^2+y) (x^2-y)^3 \in \C[x,y]$ and consider the cyclic quotient space
$M = X(2;1,1)$. Then $\{ f = 0 \} \subset M$ defines a zero set, i.e.~the
condition $(1)$ holds, but there are no $k \geq 1$ and $h \in
\Gamma(B,\cO_B^{*})$ such that $h f^k$ is a well-defined function over $M$.
\end{ex}

\begin{prop}\label{computation_ord_V_F}
Let  $B \subset \C^n$ be an open ball and $G$ a finite subgroup of $GL(n,\C)$
acting on $B$. Let $V \subset B/G =: U$ be an irreducible subvariety of
codimension one and consider $$F = \prod_{\sigma \in G} f^{\sigma}$$ where $f:B
\to \C$ is a holomorphic function defining $V$.

If $G$ is small and $f$ is chosen so that $f_x \in \cO_{B,x}$ is reduced $\forall
x \in B$, then
$
\ord_V(F:U \to \C) = |G|.
$
\end{prop}

\begin{proof}
Choose $[P] \in V$ such that $[P]$ is smooth in $U$ and $(V,[P])$ defines an
irreducible germ; then $\ord_V(F) = \ord_{V,[P]}(F)$, see~\ref{order_function}.

By Theorem~\ref{th_Prill}, since $G$ is small and $[P] \in U$ is smooth, using
the covering $\pi: B \to U$, one finds an isomorphism of germs $(U,[P]) \cong
(B/G_P,[P]) = (B,P)$ induced by the identity map.
The germ $(V,[P])$ is converted under this isomorphism into $(W,P)$ where $W$ is
the zero set of $f_P \in \cO_{B,P}$.

By Proposition~\ref{prop_not_well-defined_holomorphic}, there is $h \in
\Gamma(B,\cO_B^{*})$ such that $F = h f^{|G|}$. Then the required order is
$
\ord_{V,[P]}(F:U \to \C) = \ord_{V(f_P),P}(h f^{|G|}:B \to \C) = |G|
$
as claimed.
\end{proof}

\begin{nothing}\label{how_to_write_summarize}
Here we summarize how to write a Weil divisor as a $\mathbb{Q}$-Cartier divisor
where $X$ is an algebraic $V$-manifold.
\begin{enumerate} \setlength{\itemsep}{3pt}
\item Write $D = \sum_{i \in I} a_i [V_i] \in \WeDiv(X)$, where $a_i \in \Z$ and
$V_i \subset X$ irreducible. Also choose $\{U_j\}_{j\in J}$ an open covering of
$X$ such that $U_j = B_j/G_j$ where $B_j \subset \C^n$ is an open ball and $G_j$
is a {\bf small} finite subgroup of $GL(n,\C)$.

\item For each $(i,j) \in I \times J$ choose a reduce polynomial $f_{i,j}: U_j
\to \C$
such that $V_i \cap U_j = \{ f_{i,j} = 0 \}$. 
$$
[V_i|_{U_j}] = \frac{1}{|G_j|} \{ (U_j, f_{i,j}^{|G_j|}) \}
$$

\item Identifying $\{(U_j, f_{i,j}^{|G_j|})\}$ with its image $\CaDiv(U_j)
\hookrightarrow \CaDiv(X),$ one finally writes $D$ as a sum of locally principal
Cartier divisors over $\mathbb{Q}$,
$$
  D = \sum_{(i,j) \in I \times J} \frac{a_i}{|G_j|} \{ (U_j, f_{i,j}^{|G_j|})
\}.
$$
\end{enumerate}
\end{nothing}

We finish this section with an example where the exceptional divisor of a
weighted blow-up (which is in general just a Weil divisor) is written explicitly
as a $\mathbb{Q}$-Cartier divisor.

\begin{ex}\label{how_to_write_example}
Let $X$ be a surface with abelian quotient singularities. Let $\pi: \widehat{X}
\to X$ be the weighted blow-up at a point of type $(d;a,b)$ with respect to
$\w=(p,q)$. In general the exceptional divisor $E:=\pi^{-1}(0)\cong \bP^1_{\w}
(d;a,b)$ is a Weil divisor on $\widehat{X}$ which does not correspond to a
Cartier divisor. Let us write $E$ as an element in $\CaDiv(\widehat{X})
\otimes_{\Z} \mathbb{Q}$.

As in~\ref{blow-up2-sing-ab}, assume $\pi := \pi_{(d;a,b),\w}:
\widehat{X(d;a,b)}_{\w} \to X(d;a,b)$. Assume also that $\gcd(p,q)=1$ and
$(d;a,b)$ is normalized. Using the
notation introduced in~\ref{blow-up2-sing-ab}, the space $\widehat{X}$ is covered as $\widehat{U}_1
\cup \widehat{U}_2$ and the first chart is given by
\begin{equation}\label{isom_1st_chart_dim2}
\begin{array}{rcl}
Q_1 := X \Big( \frac{pd}{e}; 1, \frac{-q+\beta pb}{e} \Big) & \longrightarrow &
\widehat{U}_1, \\[0.25cm] \,\big[ (x^e,y) \big] & \mapsto & \big[ ((x^p,x^q
y),[1:y]_{\w}) \big]_{(d;a,b)},
\end{array}
\end{equation}
where $e := \gcd(d,pb-qa)$, see~\ref{blow-up2-sing-ab} for the details.

In the first chart $E$ is the Weil divisor $\{x=0\} \subset Q_1$. Note that the
type representing the space $Q_1$ is in a normalized form and hence the
corresponding subgroup of $GL(2,\C)$ is small.

Following the discussion~\ref{how_to_write_summarize}, the divisor $\{x=0\}
\subset Q_1$ is written as an element in $\CaDiv(Q_1) \otimes_{\Z} \mathbb{Q}$
like $\frac{e}{pd} \{ (Q_1,x^{\frac{pd}{e}}) \}$, which is mapped to
$\frac{e}{pd} \{ (\widehat{U}_1, x^d) \} \in \CaDiv(\widehat{U}_1) \otimes_{\Z}
\mathbb{Q}$ under the isomorphism~(\ref{isom_1st_chart_dim2}).

Analogously $E$ in the second chart is $\frac{e}{qd} \{ ( \widehat{U}_2,y^d)\}$.
Finally one writes the exceptional divisor of $\pi$ as claimed,
$$
E  = \frac{e}{dp} \big\{ (\widehat{U}_1, x^d), (\widehat{U}_2, 1) \big\} + 
\frac{e}{dq} \big\{ (\widehat{U}_1, 1), (\widehat{U}_2, y^d) \big\} =
\frac{e}{dpq} \big\{ (\widehat{U}_1, x^{dq}), (\widehat{U}_2, y^{dp}) \big\}.
$$
\end{ex}

\begin{defi}
Let $X$ be a $V$-manifold. The vector space of $\mathbb{Q}$-Cartier divisors is
identified under $T_X$ with the vector space of $\mathbb{Q}$-Weil divisors. A
{\em $\mathbb{Q}$-divisor} on $X$ is an element in $\CaDiv(X) \otimes_{\Z}
\mathbb{Q} = \WeDiv(X) \otimes_{\Z} \mathbb{Q}.$ The set of all
$\mathbb{Q}$-divisors on $X$ is denoted by $\mathbb{Q}$-$\Div(X)$.
\end{defi}

\section{Conclusions and Future Work}

In order to define the intersection number between two divisors the following
fact is essential. Let $X$ be a compact connected Riemann surface and let
$\pi:E\to X$ be a complex line bundle. Such a line bundle admits
meromorphic sections and the degree of such a section (sum of the order
of the zeroes minus the sum of the order of the poles) depends only
on $E$. In order to define the intersection number between two divisors
in a $V$-surface (one with compact support), we can assume (up to multiplication
by an integer) that they are Cartier divisors. Then, it is possible to define
the intersection number as the degree of the pull-back in the compact divisor of the line bundle
associated to the other divisor. These ideas will be developed in a \cite{kjj-inter}.
The following example illustrates these ideas.

\begin{ex}\label{example_X211_intersection}
Let $X=X(2;1,1)$ and consider the Weil divisors $D_1 = \{x=0\}$ and $D_2 = \{y=0\}$. 
Let us compute the Weil divisor associated with $j^{*}_{D_1} D_2$, where $j_{D_1}: |D_1| \hookrightarrow X$ 
is the inclusion. Following~\ref{how_to_write_summarize}, the divisor $D_2$ can be written as $\frac{1}{2} \{(X,y^2)\}$. 
By definition, since $|D_1| \nsubseteq |D_2|$, the pull-back is
$
  j^{*}_{D_1} D_2 = \frac{1}{2} \big\{(D_1, y^2|_{D_1}) \big\},
$
and its associated Weil divisor is 
$$
  T_{D_1} ( j^{*}_{D_1} D_2 )  = 
\frac{1}{2} \sum_{P \in D_1} \ord_P ( y^2|_{D_1} ) \cdot [P]  = 
\frac{1}{2} \ord_{[(0,0)]} ( y^2|_{D_1} ) \cdot [(0,0)] \ =  \ \frac{1}{2} \cdot [(0,0)].
$$
Note that there is an isomorphism $D_1 = X(2;1) \simeq \C$, $[y] \mapsto y^2$, 
and the function $y^2: D_1 \to \C$ is converted into the identity map $\C \to \C$ 
under this isomorphism. Hence $\ord_{[(0,0)]} ( y^2|_{D_1} ) = 1$. 
It is natural to define the (global and local) intersection multiplicity as
$
D_1 \cdot D_2 = (D_1 \cdot D_2)_{[(0,0)]} = \frac{1}{2}.
$
\end{ex}

\bibliographystyle{amsplain}
\bibliography{./references/ref_PhD}

\vspace{0.5cm}

\end{document}